\documentclass[final,11pt, english]{amsart}
\usepackage{xcolor,latexsym,amsmath,amsthm,amsfonts,amssymb,lipsum}
\usepackage[utf8]{inputenc}
\usepackage{layout,verbatim,setspace}
\usepackage{times,cite}
\usepackage{enumerate}
\def\be{\begin{equation}}
\def\ee{\end{equation}}
\def\ba*{\begin{eqnarray*}}
\def\ea*{\end{eqnarray*}}
\newcommand{\vG}{\varGamma}
\newcommand{\ve}{\varepsilon}
\newcommand{\wt}{\widetilde}

\newcommand{\vO}{\varOmega}
\newcommand{\vS}{\varSigma}

\newcommand{\ov}{\overline}

\newcommand{\N}{\mathbb{N}}
\newcommand{\R}{\mathbb{R}}
\newcommand{\mcL}{{\mathcal L}}
\newcommand{\mcS}{{\mathcal S}}
\newcommand{\mcI}{{\mathcal I}}

\newcommand{\mcH}{\mathcal H}

\newcommand{\emp}{\emptyset}

\newcommand{\vPh}{\varPhi}
\newcommand{\mcM}{\mathcal M}

\def\vD{\varDelta}

\providecommand{\keywords}[1]{\textbf{\textit{Index terms---}} #1}

\newtheorem{thm}{Theorem}[section]

\newtheorem{rem}[thm]{Remark}
\newtheorem{prop}[thm]{Proposition}

\newtheorem{prob}[thm]{Problem}
\begin{document}
\begin{flushright}
{\em \small Dedicated to Prof. Domenico Candeloro:\\No one dies on Earth, as long as he lives in the heart of those who remain\\  \dag \,  May, 3, 2019\\ }
\end{flushright}

\title[Integration of multifunctions with closed  convex ...]{Integration of multifunctions with closed \\ convex values in arbitrary Banach spaces
}

\author{D. Candeloro}
\address{Department of Mathematics and Computer Sciences - University of Perugia,  06123 Perugia (Italy), Orcid Id: 0000-0003-0526-5334}
\email{domenico.candeloro@unipg.it}
\author{L. Di Piazza}
\address{Department of Mathematics and Computer Sciences, University of Palermo,  Via Archirafi 34, 90123 Palermo (Italy), Orcd Id: 0000-0002-9283-5157 }
\email{luisa.dipiazza@unipa.it}
\author{ K. Musia{\l}}
\address{Institut of Mathematics, Wroc{\l}aw University,  Pl. Grunwaldzki  2/4, 50-384 Wroc{\l}aw (Poland), Orcid Id: 0000-0002-6443-2043}
 \email{musial@math.uni.wroc.pl}
\author{A. R. Sambucini}
\address{Department of Mathematics and Computer Sciences - University of Perugia, 06123 Perugia (Italy) , Orcid Id: 0000-0003-0161-8729}
 \email{anna.sambucini@unipg.it}
\thanks{\\
This research was partially supported by Grant  "Metodi di analisi reale per l'appros\-simazione attraverso operatori discreti e applicazioni” (2019) of GNAMPA -- INDAM (Italy), by University of Perugia --  Fondo Ricerca di Base 2019 ``Integrazione, Approssimazione, Analisi Nonlineare e loro Applicazioni'', by University of Palermo  --  Fondo Ricerca di Base 2019 and by Progetto Fondazione Cassa di Risparmio cod. nr. 2018.0419.021 (title: Metodi e Processi di Intelligenza artificiale per lo sviluppo di una banca di immagini mediche per fini diagnostici (B.I.M.)).
}
\begin{abstract}

 Integral properties of multifunctions with closed convex values  are studied.
 In this more general framework not all the tools and the technique used for weakly compact
convex valued multifunctions work. We prove that positive Denjoy-Pettis integrable multifunctions are Pettis integrable
and we obtain
 a full description of McShane  integrability in terms of Henstock and Pettis integrability, finishing the problem started by Fremlin \cite{f1994}.
\end{abstract}
 \maketitle

%===============================

\date{today}

\section{Introduction}

In the last decades, many researchers have investigated properties of measurable and integrable multifunctions and all this has been done, both because it has applications in Control Theory,  Multivalued Image Reconstruction and Mathematical Economics and because this study
is also interesting from the point of view of the  measure  and integration theories, as shown in the articles
\cite{bcs2015,BDpM2,dpmm,ncm,dp-ma,dm2,dpm-ill, Gor,ci1,ci2,balder,ckr1,ckr2,dp,yes,ccgs,ccgs2018,ccgis,bs,bo}.
\\
In particular, we believe that  comparison among  different generalizations of  Lebe\-sgue integral is one of the most fruitful areas of research in the modern theory of integration.

The choice to introduce these weaker types of integrals is motivated moreover by the fact that the well known Kuratowski and Ryll-Nardzewski Theorem requires the separability of the range space $X$,   to guarante the existence of measurable selectors.
Extensions of this theorem for weaker integrals are found for example in the articles \cite{ckr0,ckr1,ckr2,mu15} for the Pettis multivalued integral
with values in non separable Banach spaces and  \cite{dm2,cdpms2016,cdpms2016a,cdpms2016b,cdpms2019,cdpms2019a}, where the existence of  integrable selections in the same sense of the corresponding multifunctions has been considered for some gauge integrals  in the hyperspace  $cwk(X)$  ($ck(X)$) of convex and weakly compact (compact) subsets  of a general Banach space $X$.\\
%=====================
\begin{comment}
 Moreover, when the multifunction is $cwk(X)$-valued, the relations among the gauge integrals and the Aumann integral with gauge integrable selections were given, for example,  in  \cite{ckr0,ckr3} for the Pettis integral, and,   in the case of separable spaces, in \cite{ckr3,mu9} for the Gelfand and Dunford integrals, in \cite{cr1} for the Birkhoff integral and in \cite{bs} for the McShane integrability.
Some other results are also obtained in an arbitrary Banach space for $cwk(X)$-valued multifunctions in \cite{cdpms2016} and for $ck(X)$-valued multifunctions in \cite{dp}.\\
\end{comment}
%===============================
The connection between Aumann-Pettis integral and Pettis integral is well presented in \cite{eh}. If a multifunction takes as its values closed convex and bounded sets, then it is unknown whether is has a Pettis integrable selection. Consequently, whether it is Aumann-Pettis integrable. If a multifunction is Aumann-Pettis integrable, then it is Pettis integrable in a more general sense (see \cite{eh}). More precisely, instead of integrability of the support functions of the multifunction one requires only integrability of  the negative components of the support functions. Some comments are placed after Proposition \ref{p12}. Moreover, results in this direction could be found in \cite{ckr0,ckr3,mu9,cr1,bs,dp,cdpms2016}.\\

In this work, inspired by \cite{ncm, f1994,f1995,gm,gordon-p,cdpms2016, cdpms2016b,cdpms2016a,m09%,cdpms2019
},  we
 study the topic of closed convex multifunctions and  we examine two groups of integrals: those functionally determined (we call them ``scalarly defined integrals''), as   Pettis,
Henstock--Kurzweil--Pettis, Denjoy--Pettis integrals,   and those identified by
gauges (we call them ``gauge defined integrals'') as  Henstock, McShane and  Birkhoff integrals.   The last class also includes versions of Henstock and McShane integrals (the $\mcH$ and $\mcM$ integrals, respectively),    when only measurable gauges are allowed, and the variational Henstock integral.\\
In section 3 we study  properties   of scalarly defined  integrals.
  The main results of this section are
 Theorem \ref{p1} and Theorem \ref{T1}.  The first one  is a multivalued version of the well known fact that each non negative  real-valued Henstock-Kurzweil integrable function is Lebesgue integrable.\\
 In section 4 we study properties of gauge integrals. The main results are Theorems \ref{p8}, \ref{c1} and \ref{T11}, where we prove that a %scalarly integrable
 multifunction is McShane (resp. Birkhoff) integrable in $cb(X)$ if and only if it is strongly Pettis integrable and Henstock (resp. $\mcH$) integrable. If $c_0\varsubsetneq{X}$, then strong Pettis integrability may be replaced by ordinary Pettis integrability. These results completely describe the relation between Pettis and Henstock integrability and generalize our earlier achievements in this direction, when integrable multifunctions were assumed to take compact \cite{dp} or weakly compact values \cite{cdpms2016b}.\\

\section{Definitions, terminology}
Throughout   $X$ is a Banach space with its dual $X^*$.
The closed unit ball of $X$ is denoted by $B_X$. The symbol
 $c(X)$ denotes the collection of all
nonempty closed convex subsets of $X$ and $cb(X),\, cwk(X)$ and $ck(X)$  denote respectively
the family of all bounded, weakly compact and compact members of $c(X)$.  For every $C \in c(X)$ the
{\it support function of}\;  $C$ is denoted by $s( \cdot, C)$ and
defined on $X^*$ by $s(x^*, C) = \sup \{ \langle x^*,x \rangle \colon  \ x
\in C\}$, for each $x^* \in X^*$. $|C|:=\sup\{\|x\|: x\in{C}\}$  and $d_H$ is the Hausdorff metric on the hyperspace $cb(X)$. $\sigma(X^*,X)$ is the weak$^*$ topology on $X^*$ and $\tau(X^*,X)$ is the Mackey topology on $X^*$. $\mcI$ is the collection of all closed subintervals of the unit interval $[0,1]$.
 The sup norm in the space of bounded real-valued functions is denoted by $\|\cdot\|_{\infty}$.\,  All functions investigated  are defined on the unit interval $[0,1]$
 endowed with Lebesgue measure $\lambda$.
The family
of all Lebesgue measurable subsets of [0, 1] is denoted by
$\mathcal{L}$.\\
A map $\vG\colon [0,1]\to c(X)$ is called a {\it multifunction}.
$\vG$ is {\it simple} if  there exists a finite decomposition $\{A_1,..., A_p\}$ of $[0,1]$ into measurable pairwise disjoint subsets of $[0,1]$  such that $\vG$ is constant on each $A_j$. \\
$\vG:[0,1]\to{ck(X)}$ is {\it  determined by a function} $f:[0,1]\to{X}$ if $\vG(t)={\rm conv}\{0,f(t)\}$ for every $t\in[0,1]$.\\
$\vG\colon [0,1] \to{c(X)}$ is {\it positive} if $s(x^*,\vG)\geq 0$ a.e. for each $x^*\in{X^*}$ separately.\\
$\vG\colon[0,1]\to c(X)$ is said to be {\it scalarly measurable} (resp. {\it scalarly integrable}) if for every $ x{}^* \in  X{}^*$, the map
$s(x{}^*,\vG(\cdot))$ is measurable (resp. integrable).

If a multifunction is a function, then we use the traditional name of strong measurability instead of Bochner measurability.\\
A map $M\colon   \mathcal{L} \to cb(X)$ is  {\it additive}, if $M(A\cup{B})=M(A) \oplus {M(B)}$ for every pair of disjoint elements of $\vS$. An additive map $M\colon \mathcal{L} \to cb(X)$  is called a {\it  multimeasure}  if
$s(x^*,M(\cdot))$ is a finite measure, for every $x^*\in{X^*}$.  If $M$ is a
point map, then we talk about measure. If $M\colon \mathcal{L} \to cb(X)$ is  {\it countably additive} in
the Hausdorff metric (that is, if   $E_n,\, n \in \mathbb{N}$, are pairwise disjoint measurable subsets of $[0,1]$, then $$\lim_nd_H\biggl(\sum_{k=1}^nM_{\vG}(E_k),M_{\vG}\biggl(\bigcup_{k=1}^{\infty}E_k\biggr)\biggr)=0),$$  then it is called an  {\em h}-multimeasure.

It is known that if
 $M\colon \mathcal{L} \to cwk(X)$, then $M$ is a multimeasure if and only if it is an {\em h}-multimeasure (cf. \cite[Chapter 8, Theorem 4.10]{hp}).

We divide multiintegrals into two groups: functionally (or scalarly) defined integrals (Pettis, weakly McShane, Henstock-Kurzweil-Pettis and Denjoy-Pettis) and gauge integrals (Bochner, Birkhoff, McShane, Henstock, $\mcH$ and variationally Henstock).

   We remind that a scalarly integrable multifunction $\vG:[0,1]\to c(X)$ is {\it Dunford integrable} in a non-empty family $\mathcal C\subset{c(X^{**})}$,   if
for every  set $A \in \mcL$ there
exists a set $M_{\vG}^D(A)\in {\mathcal C}$  such that
\begin{equation}\label{Du}
s(x^*,M_{\vG}^D(A))=\int_A s(x^*,\vG)\,d\lambda\,,\;\mbox{for every}\; x^*\in X^*.
\end{equation}
Then $M_{\vG}^D(A)$ is called the {\em Dunford integral} of $\vG$ on $A$.
\\
If $M_{\vG}^D(A)\subset{X}$ for every $A\in\mcL$, then $\vG$ is called {\it Pettis integrable in} $\mathcal C$.
 We write then $M_{\vG}(A)$ instead of $M_{\vG}^D(A)$,
and set $(P)\int_A\vG\,d\mu :=M_{\vG}(A)$. We call $M_{\vG}(A)$ the {\it Pettis} {\it integral of} $\vG$ over
	$A$. It follows from
	the definition  that $M_{\vG}$ is a multimeasure that is $\mu$-continuous.
	We say that a Pettis integrable $\vG\colon \vO\to c(X)$ is {\it strongly Pettis
		integrable}, if $M_{\vG}$ is an $h$-multimeasure. $\mathbb{P}(\mathcal{C})$ denotes multifunctions that are Pettis integrable in $\mathcal{C}$, while  $\mathbb{P}_S(\mathcal{C})$ denotes multifunctions strongly Pettis
 integrable in $\mathcal{C}$.
\\
We recall moreover the definition of the  Denjoy integral in the wide sense
(\cite[Definition 11]{gordon-p}), called also the Denjoy-Khintchine integral,  for a real valued function. Namely, a function $f:[0,1]\to\R$ is {\em Denjoy integrable in the wide sense}, if there exists an ACG function (cf. \cite{Gor})  F such that its approximate derivative is almost everywhere equal to $f$. For simplicity, we call such a function Denjoy integrable and use the symbol $(D)\int f$.

A multifunction $\vG:[0,1]\to{c(X)}$ is said to be {\em  Denjoy-Pettis} (or DP) {\em integrable } in ${\mathcal C}\subset{c(X)}$,   if it is scalarly Denjoy integrable and for each $I\in\mcI$ there exists a set $N_{\vG}(I)\in{\mathcal C}$ such that
\begin{eqnarray}\label{DP}
s(x^*,N_{\vG}(I))=(D)\int_Is(x^*,\vG)\quad \mbox{ for every }  x^*\in{X^*}.
\end{eqnarray}

If in the previous definition,   the multifunction $\Gamma$ is
 scalarly Henstock-Kurzweil (or HK) integrable  we say  that the multifunction $\vG$ is  {\em Henstock-Kurzweil-Pettis} (or HKP) {\em  integrable} in ${\mathcal C}$.
The family of all DP-integrable (resp. HKP-integrable)  multifunctions in $\mathcal C$ is denoted by ${\mathbb {DP}}({\mathcal C})$ (resp. ${\mathbb {HKP}}({\mathcal C}))$.
\\
If an HKP-integrable multifunction $\vG$ is also scalarly integrable, then it is called {\it weakly McShane}  (or $w$MS) {\em integrable}. The family of all $w$MS-integrable functions in $\mathcal C$ is denoted by ${w\mathbb {MS}}({\mathcal C})$.
\\
   Moreover, a  weak$^*$ scalarly integrable multifunction $\vG:[0,1]\to c(X^*)$ is {\it Gelfand integrable} in $\mathcal C\subset{c(X^*)}$,  if for each  set $A \in \mcL$ there
exists a set $M_{\vG}^G(A)\in {\mathcal C}$  such that
\begin{equation}\label{Ge}
s(x,M_{\vG}^G(A))=\int_A s(x,\vG)\,d\lambda\,,\;\mbox{for every}\; x\in X.
\end{equation}
$M_{\vG}^G(A)$ is called the {\em Gelfand integral} of $\vG$ on $A$. \\

For the gauge integrals we need some preliminary definitions and
to avoid misunderstanding let us
 point out that gauge integrable multifunctions take always bounded values ($d_H$ is well defined on bounded sets only) whereas scalarly defined integrals integrate multifunctions with arbitrary closed convex values. \\
 A {\it partition} ${\mathcal P}$ {\it in} $[0,1]$ is a collection $\{(I{}_1,t{}_1),$ $ \dots,(I{}_p,t{}_p) \}$,
 where $I{}_1,\dots,I{}_p$ are nonoverlapping subintervals of $[0,1]$, $t{}_i$ is a point of $[0,1]$, $i=1,\dots, p$.
 If $\cup^p_{i=1} I{}_i=[0,1]$, then  ${\mathcal P}$ is {\it a partition of } $[0,1]$. If   $t_i$ is a point of $I{}_i$, $i=1,\dots,p$,  we say that ${\mathcal P}$ is a {\it Perron partition of } $[0,1]$.\\
 A {\it gauge} on $[0,1]$ is a positive function on $[0,1]$. For a given gauge $\delta$ on $[0,1]$,
 we say that a partition $\{(I{}_1,t{}_1), \dots,(I{}_p,t{}_p)\}$ is $\delta$-{\it fine} if
 $I{}_i\subset(t{}_i-\delta(t{}_i),t{}_i+\delta(t{}_i))$, $i=1,\dots,p$.\\
A multifunction $\vG:[0,1]\to cb(X)$ is said to be {\it Henstock} (resp. {\it McShane})
  {\it  integrable} on $[0,1]$,  if there exists   $\vPh{}_{\vG}([0,1]) \in cb(X)$
    with the property that for every $\varepsilon > 0$ there exists a gauge $\delta$ on $[0,1]$
such that for each
$\delta$-fine Perron partition (resp. partition)
we have
\begin{eqnarray}\label{H-MS}
d{}_H \biggl(\vPh{}_{\vG}([0,1]),\sum_{i=1}^p\vG(t{}_i)|I{}_i|\biggr)<\ve\,.
\end{eqnarray}
$\vG$ is said to be {\it Henstock} (resp. {McShane})
  {\it  integrable} on  $I \in \mathcal{I}$ (resp. $E\in
 \mathcal{L}$) if
$\vG 1_I$
(resp. $\vG 1{}_E$) is  integrable on $[0,1]$
in the corresponding sense. Moreover, if the gauge $\delta$ of the Henstock integrability is measurable we speak on $\mathcal{H}$-integrability (see also \cite{ncm}).\\
A multifunction $\vG:[0,1]\to cb(X)$ is said to be {\it Birkhoff}
   integrable on $[0,1]$,
   if it is McShane integrable but the gauges are measurable functions.
As before,
 we denote by $\mathbb{H}(cb(X)) \, (resp.    \mathcal{H}(cb(X)), \, \mathbb{MS}(cb(X)),\, \mathbb{BI}(cb(X))$), the spa\-ces of Henstock, (resp. Henstock with measurable gauges,  McShane, Birkoff) integrable multifunctions in $cb(X)$.
\\
   A multifunction $\vG:[0,1]\to cwk(X)$ is said to be {\it variationally Henstock} ({\it
resp. McShane})
   integrable,
   if there exists
   a multimeasure  $\vPh_{\vG}: {\mathcal I} \to {cb(X)}$
   (resp. $\vPh_{\vG}: {\mathcal{L}} \to {cb(X)}$) with the following property:
   for every $\ve>0$ there exists a gauge $\delta$
   on $[0,1]$ such that for each $\delta$-fine Perron partition (resp.  partition)
   $\{(I_1,t_1), \dots,(I_p,t_p)\}$
we have
\begin{eqnarray}\label{aa}
\sum_{j=1}^pd_H \left(\vPh_{\vG}(I_j),\vG(t_j))|I_j|\right)<\ve\,.
\end{eqnarray}
 The set  multifunction  $\vPh_{\vG}$  will be  called the {\it variational Henstock} ({\it McShane}) {\it primitive} of $\vG$.
\\
\begin{comment}{\color{blue!50!green}
We recall that a selection of a multifunction $\Gamma$ is a single valued function $f$ such that $f(t) \in \Gamma(t)$ for all $t \in [0,1]$.\\
An almost selector $f$ for the multifunction $F$ is such that $x^* f \leq s(x^*, F)$ a.e. for every $x^*$ (with the exception set depending on $x^*$).\\
}
\end{comment}
Finally  ${\mathcal{S}}_H(\vG)\;(resp.\, {\mathcal{S}}_{MS}(\vG)\,
   \mathcal{S}_P(\vG),\,  {\mathcal{S}}_{HKP}(\vG)\,,{\mathcal{S}}_B(\vG)\,,\,  {\mathcal{S}}_{vH}(\vG)\,,\;...)$ denotes the family of all scalarly measurable selections of $\vG$
    that are Henstock (resp. McShane, Pettis, Henstock-Kurzweil-Pettis, Birkhoff, variationally Henstock, ...) integrable.\\

A useful tool to study the $cb(X)$-valued multifunctions
 is the R{\aa}dstr\"{o}m embedding (see, for example, \cite[Theorem 3.2.9 and Theorem 3.2.4(1)]{Beer}
or \cite[Theorem II-19]{CV})
 $i:cb(X)\to l_{\infty}(B_{X^*})$
 given by $i(A):=s(\cdot, A)$. It  satisfies the following properties:
\begin{description}
\item[1)] $i(\alpha A \,  \oplus \,  \beta C) = \alpha i(A) + \beta i(C)$ for every $A,C\in  cb(X),\,\, \alpha, \beta \in  \mathbb{R}{}^+$;  (here the symbol $\oplus$ is the Minkowski addition)
\item[2)] $d_H(A,C)=\|i(A)-i(C)\|_{\infty},\quad A,C\in  cb(X)$;
\item[3)] $i(cb(X))$  is a closed cone in the space
$l_{\infty}(B_{X^*})$ equipped with the norm of the uniform convergence.
\end{description}
\section{Scalarly defined integrals.}
 The following result is a generalization of \cite [Theorem 6.7]{mu8}.
\begin{thm}\label{70}
If $\vG:[0,1]\to c(X^*)$ is weak$^*$ scalarly integrable, then
$\vG$ is Gelfand integrable in $cw^*k(X^*)$.
\end{thm}
\begin{proof}
 Assume at the beginning that $\vG$ is weak$^*$ scalarly bounded (i.e.  there exists $0<K<\infty$  such that $|s(x,\vG)|\leq{K}\|x\|$ a.e. for each $x\in{X}$ separately).
Let us fix $A\in\mcL$ and define a  sublinear functional on $X$ setting
$\varphi_A(x):\,=\int_As(x,\vG)\,d\lambda.$
One can easily see that $\varphi_A$ is norm continuous.
This proves  the existence of a set $C_A\in{cw^*k(X^*)}$ such that $\varphi_A(x)=s(x,C_A)$, for every $x\in{X}$  (we simply take as $C_A$ the weak$^*$-closure of the set $\{x^*\in{X^*}\colon \langle{x^*,x}\rangle\leq \varphi_A(x)\}$).
 Consequently, $\vG$ is Gelfand integrable in $cw^*k(X^*)$.
The general case follows by decomposition of $\vG$ in a series of weak$^*$ scalarly bounded multifunctions (see \cite [Theorem 6.7]{mu8}).
\end{proof}
As a direct consequence of Theorem \ref{70} we obtain the following generalization of \cite[Theorem 6.9]{mu8}  to the case of $c(X)$ valued multifunctions:
\begin{thm}\label{t14}
Each scalarly integrable multifunction $\vG:[0,1]\to{c(X)}$ is Dunford integrable in $cw^*k(X^{**})$.
\end{thm}

  In \cite[Proposition 23]{yes} an example of a {$w$}MS-integrable function
   is given which is not Pettis integrable. The same property has the function constructed in \cite{gm}.
 In case of positive multifunctions  the situation is different.\\

The next result has been proven in \cite[Lemma 1 and Remark  3]{dm2} for the Denjoy--Pettis integral and multifunctions with weakly compact values. Unfortunately that proof fails in the general case.
\begin{thm}\label{p1}
If $\vG \in{\mathbb {DP}}(cb(X))$ (resp. ${\mathbb {DP}}(cwk(X))$, \;${\mathbb {DP}}(ck(X))$)  is a positive multifunction, then
 $\vG\in{\mathbb {P}}(cb(X))$ (resp. ${\mathbb {P}}(cwk(X))$, \;${\mathbb {P}}(ck(X))$).
 \end{thm}
\begin{proof} Assume that $\vG\in{\mathbb {DP}}(cb(X))$.
Since $s(x^*,\vG)$ is a.e. non-negative and Denjoy integrable,
 it is  Lebesgue integrable (cf. \cite[Theorem 7.7]{Gor}). By the assumption, for every $I\in\mcI$ there exists $N_{\vG}(I)\in{cb(X)}$ such that
$$
s(x^*,N_{\vG}(I))=\int_Is(x^*,\vG)\,d\lambda\quad\mbox{for every}\;x^*\in{X^*}.
$$
In virtue of Theorem \ref{t14} $\vG$ is Dunford integrable in $cw^*k(X^{**})$:
$$
\forall\;E\in\mcL\;\exists\;M_{\vG}^D(E)\in{cw^*k(X^{**})}\;  \forall\;x^*\in{X^*}\;s(x^*,M_{\vG}^D(E))=\int_Es(x^*,\vG)\,d\lambda.
$$
Thus, for every $I\in\mcI$ we have the equality
$
s(x^*,N_{\vG}(I))=s(x^*,M_{\vG}^D(I)).
$
Due to the Hahn-Banach theorem, it follows $M_{\vG}^D(I)=\ov{N_{\vG}(I)}^*$ and $X\cap{M_{\vG}^D(I)}=N_{\vG}(I)$, for every $I\in\mcI$.
We are going to prove that $\vG$ is Pettis integrable. So let us fix  $E\in\mcL$. Since the support functionals are a.e. non-negative, we have 
$M_{\vG}^D(E)\subset M_{\vG}^D[0,1]$ and then  ${X}\cap M_{\vG}^D(E)\subset {X}\cap M_{\vG}^D[0,1]=N_{\vG}[0,1]$. 
The set $N_{\vG}(E):={X}\cap M_{\vG}^D(E)$ is closed and 
$$
s(x^*,N_{\vG}(E))=s(x^*,\ov{{X}\cap M_{\vG}^D(E)}^*)\leq s(x^*,M_{\vG}^D(E))\,.
$$ 
Consequently, we have
\begin{eqnarray}\label{e20}
s(x^*,N_{\vG}(E))\leq\int_Es(x^*,\vG)\,d\lambda\quad\mbox{for all}\;x^*\in{X^*}.
\end{eqnarray}
But $s(x^*,N_{\vG}):\mcL\to\R$ is an additive set function that is, due to the inequality (\ref{e20}) countably additive. Since  both sides of (\ref{e20}) coincide on $\mcI$, they coincide on $\mcL$ and (\ref{e20}) becomes equality. In this way we obtain the required Pettis integrability of $\vG$ in $cb(X)$.
\end{proof}
 A useful application of   above property for   positive multifunctions is the  decomposition of  a multifunction ${\vG}$ integrable in ``a certain sense''  into a sum of one of  its selections  integrable in the same way  and a  positive  multifunction ``integrable in a stronger sense'' than   ${\vG}$ is.
An important  key ingredient  in such a decomposition  is the  existence of  selections  ``integrable in the same sense'' as the corresponding multifunction.
  The existence of scalarly measurable selections of arbitrary weakly compact valued scalarly measurable multifunctions has been proven by Cascales, Kadets and Rodriguez in \cite{ckr1}.\\

Concerning the integrability of selections for functionally defined  multifunctions with weakly compact values   the following holds:
\begin{prop}\label{T2}
{\rm (see \cite{dm2})}
	If the multifunction  $\vG:[0,1] \to{cwk(X)}$  is DP (resp. HKP, Pettis or weakly McShane)
	integrable in $cwk(X)$, and $f$ is a scalarly measurable selection of $\vG$, then
	$f$ is respectively DP  (resp. HKP, Pettis or weakly McShane)
	integrable.
\end{prop}
In the more general case of $cb(X)$-valued multifunctions we do not know if   each   scalarly measurable multifunction possesses scalarly measurable selections.\\

Decomposition theorems in case of weakly compact valued multifunctions   have been proven in \cite[Theorem 1 and Remark  3]{dm2} and in \cite[Theorem 3.2]{cdpms2016b}. Below, we formulate the results in a more general situation.
\begin{thm}\label{T1}
If $\vG:[0,1]\to {c(X)}$ is  a %scalarly measurable
 multifunction, then the following conditions are equivalent:
\begin{enumerate}
	\item[{\upshape \ref{T1}.i)}]
	$\vG$ is $DP$-integrable  in $cb(X)$ and ${\mathcal S}_{DP}(\vG)\neq\emp$;
	\item[{\upshape \ref{T1}.ii)}]
	${\mathcal S}_{DP}(\vG)\neq\emp$ and for all $f\in {\mathcal S}_{DP}(\vG)$  the multifunction $G:[0,1]\to cb(X)$ defined by  $G=\vG-f$ is  Pettis integrable in $cb(X)$;
	\item[{\upshape \ref{T1}.iii)}]
	There exists $f\in{\mathcal S}_{DP}(\vG)$  such that the multifunction $G=\vG-f$ is  Pettis integrable in $cb(X)$;
\end{enumerate}
$DP$-integrability above may be replaced by $HKP$ or $wMS$-integrability.
\end{thm}
\begin{proof} ${\rm  \ref{T1}.i)}\Rightarrow {\rm  \ref{T1}.ii)}$ follows  by Theorem  \ref{p1} to $G:=\vG-f$. The other implications are clear.
\end{proof}
\begin{rem}\label{r1}{\rm Exactly in the same manner one proves the analogous decomposition theorems in case of multifunctions $\vG$ that are HKP-integrable or weakly McShane integrable in $cb(X)\,,cwk(X)$ or $ck(X)$. }
\end{rem}
 By the previous decompositions we obtain:
 \begin{thm}\label{T3}%\mg{T3}
 Let $\vG:[0,1]\to{c(X)}$ be a DP-integrable multifunction.
\begin{enumerate}
\item[\upshape \ref{T3}.i)]
If  $\mcS_{HKP}(\vG)\neq\emp$, then $\vG$ is HKP-integrable.
\item[\upshape \ref{T3}.ii)]
If  $\mcS_{wMS}(\vG)\neq\emp$, then $\vG$ is wMS-integrable.
\item[\upshape \ref{T3}.iii)]
If  $\mcS_P(\vG)\neq\emp$, then $\vG$ is Pettis integrable.
\end{enumerate}
\end{thm}
\begin{proof} (\ref{T3}.i) If $\vG$ is DP-integrable and $f\in\mcS_{HKP}(\vG)$, then, according to Theorem \ref{T1}, $\vG=G+f$, where $G$ is Pettis integrable. Being Pettis integrable, $G$ is also HKP integrable, what yields HKP integrability of $\vG$.
(\ref{T3}.ii) and (\ref{T3}.iii) can be proved in a similar way.
\end{proof}
In case of $cwk(X)$-valued $\vG$ and HKP integrable $\vG$, the necessary decomposition was proved in \cite[Theorem 1]{dm2}.\\

Now we are going to concentrate on a particular family
of positive multifunctions: the class of
 multifunctions that are determined by  integrable functions. Such multifunctions  quite often serve as examples and counterexamples.
It is interesting to know which properties of the function can be transferred to the generated multifunction.% (see also \cite{cdpms2019}).
\begin{prop}\label{p12}
If $\vG$ is determined by a scalarly  measurable $f$, then it is Pettis integrable in $cwk(X)$ if and only if $f$ is Pettis integrable.
\end{prop}
\begin{proof}
Observe first that $\vG$ is scalarly measurable.
If $\vG$ is Pettis integrable, then $f$ is Pettis integrable  by \cite[Corollary 2.3]{ckr2}. Viceversa, if $f$ is Pettis integrable, by  \cite[Theorem 2.6]{mu15}
$\vG$ is  Pettis integrable in $cwk(X)$, since $|s(x^*,\vG(t))| \leq |(x^*,f(t))|$.
\end{proof}
If one investigates multifunctions that are integrable in $cb(X)$ the situation is  more complicated. If $f:[0,1]\to{X}$ is strongly measurable and scalarly integrable, then the multifunction determined by $f$ is Pettis integrable in $cb(X)$ (see \cite[Theorem 3.7]{eh}).
An example of $c_0$--valued function $f$ that is not Pettis integrable but $\vG:[0,1]\to{ck(c_0)}$ defined by $\vG(t)={\rm conv}\{0,f(t)\}$ is Pettis integrable in $cb(c_0)$ can be found in \cite[Example 1.12]{mu15}.
The same example can be used to show that DP-integrability of $f:[0,1]\to{X}$ does not guarantee the DP-integrability in $cwk(X)$ of $\vG$ determined by $f$. Indeed,  it follows from Proposition \ref{p1} that $\vG\not\in  \mathbb{DP}(cwk(X))$, since otherwise $f$ would be Pettis integrable.\\

The next result is a strengthening of  \cite[Proposition 4]{dm2}
 in case of a multifunction determined by a function.
\begin{prop}\label{p4}
Let  $f:[0,1] \to{X}$ be scalarly measurable.  If  all scalarly measurable selections of  $\vG$ determined by $f$   are DP--integrable, then $\vG$ is  Pettis integrable in $cwk(X)$.
\end{prop}
\begin{proof}
If $E\in\mcL$, then $\wt{f}:[0,1]\to{X}$ defined by $\wt{f}(t)=f(t)$ if $t\in{E}$ and zero otherwise is a DP-integrable selection of $\vG$. It follows that $f$ is Pettis integrable. The assertion follows from Proposition \ref{p12}.
 \end{proof}
 \section{Gauge integrals}
 In case of positive multifunctions with weakly compact values and integrals, it has been proven   in  \cite[Propositions  3.1 and 4.1]{cdpms2016b} that Henstock (resp. $\mcH$) integrability  implies McShane (resp. Birkhoff) integrability. In the  general case of $cb(X)$ valued multifunctions,   we do not know if positive Henstock or $\mcH$-integrable multifunctions are in fact McShane or Birkhoff integrable. We do not know even if positive Pettis and Henstock or $\mcH$-integrable multifunctions are in fact McShane or Birkhoff integrable. But if we assume something on the Banach space $X$ or  we require something more on the multifunction, then the result in \cite{cdpms2016b} can be generalized.\\

 First we need  one supplementary fact.
\begin{prop}\label{p11}
 If $X$ does not contain any isomorphic copy of $c_0$, then $M:\mcL\to{cb(X)}$ is an h-multimeasure if and only if it is a multimeasure.
\end{prop}
\begin{proof} Let us notice first that the fact that $M$ is defined on $[0,1]$ endowed with Lebesgue measure is totally unimportant. It may be defined on an arbitrary measure space.

Assume that $M$ is a multimeasure and let $\{E_i:i\in\N\}$ be a sequence of measurable and pairwise disjoint sets in $[0,1]$. Take arbitrarily $x_i\in{M(E_i)},\;i\in\N$ and $x^*\in{X^*}$. If $\pi$ is a permutation of $\N$  and $m\leq{n}$, then
$$
-s\biggl(-x^*,\sum_{i=m}^nM(E_{\pi(i)})\biggr)\leq\biggl\langle{x^*,\sum_{i=m}^nx_{\pi(i)}}\biggr\rangle\leq
s\biggl(x^*,\sum_{i=m}^nM(E_{\pi(i)})\biggr)\,.
$$
It follows that the sequence $\biggl\{\sum_{i=1}^nx_{\pi(i)}\biggr\}_n$ is weakly Cauchy and consequently the series $\sum_{n=1}^{\infty}x_n$ is weakly unconditionally Cauchy. But as $c_0\nsubseteq{X}$ the series is unconditionally convergent in the norm of $X$ due to Bessaga-Pe{\l}czy\'{n}ski result \cite{bp} (cf. \cite[Theorem V.8]{di}). Set $\vD(E):=\biggl\{\sum_{i\geq 1}x_i\colon x_i\in{M(E_i)}\biggr\}$. Exactly as in the proof of Theorem \cite[Theorem 8.4.10]{hp} one can prove that $\vD(E)=M(E)$ for every $E\in\mcL$ and that will complete the whole proof.
\end{proof}

So we have:
\begin{thm}\label{p8}
Let  $\vG:[0,1]\to{cb(X)}$.
Then, $\vG \in \mathbb{MS}(cb(X))$ (resp. $\vG\in \mathbb{BI}(cb(X))$) if and only if
 $\vG \in \mathbb{P}_s(cb(X))$ and $\vG \in \mathbb{H}(cb(X))$ (resp.  $\vG\in  \mathcal{H}(cb(X))$).
 \end{thm}
 \begin{proof}
If $\vG$ is strongly Pettis integrable the range of $(P)\int\vG$ via the R{\aa}dstr\"{o}m embedding is a vector measure.  Now we follow the proof of \cite[Proposition 3.1]{cdpms2016b}.
In fact, we can observe that  $(P)\int_I\vG = (H)\int_I\vG$ for every $I\in\mcI$.  The strong integrability  guarantees the convergence of each series $\sum_n (H) \int_{I{}_n} i \circ \vG$,
where $(I{}_n){}_n$ is any sequence of pairwise non-overlapping subintervals of $[0,1]$,
since  $(H)\int_I i \circ \vG =i\circ ((H)\int_I\vG)=i\circ (P)\int_I\vG$,  for every $I\in\mcI$.
Applying now \cite[Corollary 9 (iii)]{f1994} we obtain McShane integrability of $i \circ \vG$.
 If $\vG$ is $\mcH$-integrable,  we can apply
\cite[Theorem 8]{{f1994}} and  \cite[Theorem 2.11]{cdpms2016b}.
 \end{proof}
\begin{prob}
What is the situation if $\vG$ is strongly Pettis and variationally Henstock integrable?
\end{prob}
 Even in  single valued case $\vG$ need not be variationally McShane integrable. An example is given in \cite{dp-ma}.
\begin{thm} \label{c1}
Let  $\vG:[0,1]\to{cb(X)}$.
If $c_0\nsubseteq{X}$, then  $\vG \in \mathbb{MS}(cb(X))$ (resp. $\vG\in \mathbb{BI}(cb(X))$) if and only if  $\vG \in \mathbb{P}(cb(X))$ and $\vG \in \mathbb{H}(cb(X))$ (resp.  $\vG\in  \mathcal{H}(cb(X))$).
 \end{thm}
 \begin{proof}
 If $c_0\nsubseteq{X}$ then, by Proposition \ref{p11}, $\vG \in
\mathbb{P}_s(cb(X))$. We apply Theorem \ref{p8}.
 \end{proof}
We outline that 	$\vG \in
\mathbb{P}_s(cb(X))$ or  $c_0\nsubseteq{X}$
are key ingredients  in Theorem \ref{p8} and Theorem \ref{c1}. Due to Theorem \ref{p1} we know that if  $\vG \in \mathbb{H}(cb(X))$  is positive, then it is Pettis integrable. It remains an open question if 		there exist a positive Henstock integrable multifunction $\vG:[0,1]\to{cb(c_0)}$ that is  not strongly Pettis integrable.\\

If $\vPh:\mcI\to{cb(X)}$ is an additive multifunction, then given $I\in\mcI$, the variation of $\vPh(I)$ is defined by
$$
\wt\vPh(I):=\sup\{\sum_i\|\vPh(I_i)\|\colon \{I_1,\ldots,I_n\} \;\text{is a finite partition of }I\}\,.
$$
If $\,\wt\vPh[0,1]<\infty$, then $\vPh$ is said to be of finite variation. In this case Theorem \ref{p8} has a stronger form.
\begin{thm}\label{T11}
Let $\vG:[0,1]\to{cb(X)}$ be Henstock (or $\mcH$) integrable and let $\vPh_{\vG}$ be its H ($\mcH$)-integral. If $\wt{\vPh}_{\vG}[0,1]<\infty$, then $\vG$ is McShane (or Birkhoff) integrable.
\end{thm}
 \begin{proof} By the assumption $i\circ\vG$ is Hennstock ($\mcH$) integrable.  Consequently,  if $(I_n)_n$ is a sequence of non-overlapping subintervals of $[0,1]$ then,  due to the finite variation of $\wt\vPh_{\vG}$, the series $\sum_n(H)(\mcH)\int_{I_n}i\circ\vG$ is absolutely convergent in $l_{\infty}(B_{X^*})$, hence also convergent.  Thus, $i\circ\vG$ is McShane (Birkhoff) integrable and,  this yields McShane (Birkhoff) integrability of $\vG$.
 \end{proof}

 We are going to present now a decomposition theorem for multifunctions that are H (resp. $\mcH$)-integrable in $cb(X)$. While for weakly compact valued multifunctions properly integrable selections exist (see \cite{dp} for the  Henstock  or the McShane integral, \cite{cdpms2016} for the Birkhoff or the  variational Henstock integral), we do not know if that is the case also for $cb(X)$-valued multifunctions.
 In order to obtain a decomposition of H or $\mcH$-integrable multifunction, we have to assume that the set of suitably integrable selections is non-void.\\

 Moreover, we do not know if positive Henstock or $\mcH$-integrable multifunctions are in fact McShane or Birkhoff integrable (as it was proved in \cite[Lemma 3.1 and 4.1]{cdpms2016b} for weakly compact valued multifunctions). We do not know even if positive Pettis and Henstock or $\mcH$-integrable multifunctions are in fact McShane or Birkhoff integrable. Therefore the theorem below differs from \cite[Theorem 3.3 and 4.3]{cdpms2016b}.
 \begin{thm} \label{T5}
Let $\vG:[0,1]\to {cb(X)}$ be
 multifunction such that ${\mathcal S}_{H}(\vG)\neq\emp$ (${\mathcal S}_{\mathcal H}(\vG)\neq\emp$). Then the following conditions are equivalent:
\begin{enumerate}
	\item[{\upshape \ref{T5}.i)}]
	$\vG$ is  $H$-integrable  (resp. ${\mathcal{H}}$-integrable) in $cb(X)$;
	\item[{\upshape  \ref{T5}.ii)}]
	For all $f\in {\mathcal S}_{H}(\vG)$ (resp. $f\in {\mathcal S}_{\mcH}(\vG)$), the multifunction $G:[0,1]\to cb(X)$ defined by  $G=\vG-f$ is  Pettis and Henstock integrable (resp. Pettis and $\mcH$-integrable) in $cb(X)$;
	\item[{\upshape \ref{T5}.iii)}]
	There exists such an $f\in {\mathcal S}_{H}(\vG)$ (resp. $f\in {\mathcal S}_{\mcH}(\vG)$) that the multifunction $G:[0,1]\to cb(X)$ defined by  $G=\vG-f$ is   Pettis and Henstock integrable (resp. Pettis and $\mcH$-integrable) in $cb(X)$;
\end{enumerate}
\end{thm}
\begin{proof} ${\rm \ref{T5}.i)}\Rightarrow {\rm \ref{T5}.ii)}$. If $f\in {\mathcal S}_{H}(\vG)$ (resp. $f\in {\mathcal S}_{\mathcal{H}}(\vG)$), then $G:=\vG-f$ is also H-integrable (resp. $\mcH$-integrable). It follows from Theorem \ref{p1} that $G$ is also Pettis integrable.
\end{proof}
\begin{prob}
Is each positive Pettis integrable multifunction McShane integrable?
\end{prob}
        In case of multifunctions with weakly compact values, each positive H-integrable multifunction  is McShane integrable (see \cite[Proposition 3.1]{cdpms2016b}).   Suppose that positive  H-integrable multifunctions possessing an H-integrable selection are McShane integrable, and let $\vG$ be an H-integrable multifunction possessing an
        MS-integrable selection. Then  $\vG$ can be written as $\vG=G+f$, where $f\in\mcS_{MS}(\vG)$ and $G$ is Pettis integrable. But then $G$ is also Henstock integrable. Consequently, $G$ is McShane integrable, and also  $\vG$ is.\\

 In the general case the  following questions remain an open problem:
\begin{prob}
 Is each positive  H-inte\-grable multifunction (possessing an H-inte\-grable selection)  McShane integrable?
            Is each positive Pettis integrable multifunction strongly Pettis integrable?
\end{prob}
\begin{rem} \rm
Finally it is worth to note that in  all previous  results concerning the representation of a multifunction $\varGamma$ as a sum of one of its selections and a positive multifunction,  it is sufficient to have a  {\it  quasi selection} $f$ (cf. \cite{mu15}), i.e. such a function $f$ that $x^* f \le s(x^*,\varGamma)$
a.e. for each  $x^* \in X^*$ separately. In fact, if $f$ is a quasi selection, the  multifunction $\varGamma - f$ is a.e. positive.
\end{rem}
\section*{Acknowledgement}
The authors thank Vladimir Kadets and the anonymous referee for  their suggestions concerning  the research topic.

\end{document}